\documentclass[12pt,twoside]{amsart}
\usepackage{amsmath}
\usepackage{amsthm}
\usepackage{amsfonts}
\usepackage{amssymb}
\usepackage{latexsym}
\usepackage{mathrsfs}
\usepackage{amsmath}
\usepackage{amsthm}
\usepackage{amsfonts}
\usepackage{amssymb}
\usepackage{latexsym}
\usepackage{geometry}
\usepackage{dsfont}
\usepackage[dvips]{graphicx}
\usepackage{color}
\usepackage[all]{xy}

\date{}
\allowdisplaybreaks[4] \footskip=15pt
\renewcommand{\uppercasenonmath}[1]{}

\usepackage{graphicx,amssymb}
\usepackage[all]{xy}
\usepackage{amsmath}

\allowdisplaybreaks
\usepackage{amsthm}
\usepackage{color}

\theoremstyle{plain}
\newtheorem{theorem}{Theorem}[section]
\newtheorem{proposition}[theorem]{Proposition}

\theoremstyle{definition}
\newtheorem{example}[theorem]{Example}
\newtheorem{definition}[theorem]{Definition}

\theoremstyle{definition}
\newtheorem*{acknowledgement}{Acknowledgement}

\theoremstyle{remark}
\newtheorem{remark}[theorem]{Remark}

\newcommand{\Tor}{\mbox{\rm Tor}}

\newcommand{\C}{\mathcal{C}}

\newcommand{\Id}{\mathrm{Id}}

\def\p{\frak p}
\def\m{\frak m}

\def\Hom{{\rm Hom}}
\def\Ext{{\rm Ext}}
\def\Tor{{\rm Tor}}

\def\Ker{{\rm Ker}}

\def\Im{{\rm Im}}
\def\Coker{{\rm Coker}}

\def\Max{{\rm Max}}

\def\Spec{{\rm Spec}}
\def\Max{{\rm Max}}

\def\Id{{\rm Id}}

\begin{document}
\begin{center}
{\large  \bf On uniformly $S$-absolutely pure modules}

\vspace{0.5cm}  \ Xiaolei Zhang$^{a}$

{\footnotesize a.\ \ \ School of Mathematics and Statistics, Shandong University of Technology, Zibo 255049, China\\

E-mail: zxlrghj@163.com\\}
\end{center}

\bigskip
\centerline { \bf  Abstract}
\bigskip
\leftskip10truemm \rightskip10truemm \noindent

Let $R$ be a commutative ring with identity and $S$ a multiplicative subset of $R$. In this paper, we  introduce and study the notions of $u$-$S$-pure $u$-$S$-exact sequences and uniformly $S$-absolutely pure modules which extend the classical notions of pure exact sequences and absolutely pure modules. And then we characterize uniformly $S$-von Neumann regular rings and  uniformly $S$-Noetherian rings using uniformly $S$-absolutely pure modules.
\vbox to 0.3cm{}

\leftskip0truemm \rightskip0truemm
\bigskip

\section{introduction and Preliminary}

Throughout this paper, $R$  is always a commutative ring with identity, all modules are unitary and $S$  is always a multiplicative subset of $R$, that is, $1\in S$ and $s_1s_2\in S$ for any $s_1\in S, s_2\in S$.

The notion of absolutely pure modules was first introduced by  Maddox \cite{M67} in 1967. An $R$-module $E$ is said to be \emph{absolutely pure} provided that $E$ is a pure submodule of every module which contains $E$ as a submodule. It is well-known that an $R$-module $E$ is absolutely pure if and only if $\Ext_R^1(N,E)=0$ for any finitely presented module $N$ (\cite[Proposition 2.6]{S70}). So absolutely pure modules are also studied with the terminology FP-injective modules (FP for finitely presented), see Stenstr\"{o}m \cite{S70} and Jain \cite{J73}  for example. The notion of absolutely pure modules is very attractive in that it is not only a generalization of that of injective modules but also an important tool to characterize some classical rings.  A ring $R$ is semihereditary if and only if any homomorphic image of an absolutely pure $R$-module is absolutely pure(\cite[Theorem 2]{M70}). A ring $R$ is Noetherian if and only if any absolutely pure $R$-module is injective (\cite[Theorem 3]{M70}). A ring $R$ is von-Neumann regular if and only if any $R$-module is absolutely pure (\cite[Theorem 5]{M70}).  A ring $R$ is coherent if and only if the class of absolutely pure $R$-modules is closed under direct limits, if and only if the class of absolutely pure $R$-modules is a (pre)cover (\cite[Theorem 3.2]{S70},  \cite[Corollary 3.5]{DD19}).

One of the most important methods to generalize the classical rings and modules is in terms of  multiplicative subsets $S$ of $R$ (see \cite{ad02,s19,bh18,l15,lO14} for example). In 2002,  Anderson and Dumitrescu \cite{ad02} introduced \emph{$S$-Noetherian rings} $R$ in which for any ideal $I$ of $R$, there exists a finitely generated sub-ideal $K$ of $I$ such that  $sI\subseteq K$.  Cohen's Theorem, Eakin-Nagata Theorem and Hilbert Basis Theorem for $S$-Noetherian rings are given in \cite{ad02}. However, the choice of $s\in S$ such that $sI\subseteq K$ in the definition of $S$-Noetherian rings as above is not uniform. Hence,  Qi et al. \cite{QKWCZ21} introduced the notion of uniform $S$-Noetherian rings and obtained the Eakin-Nagata-Formanek Theorem and Cartan-Eilenberg-Bass Theorem for  uniformly $S$-Noetherian rings.  Recently, the first author of the paper \cite{zwz21} introduced the notions of $u$-$S$-flat modules and uniformly $S$-von Neumann regular rings which can be seen as  uniformly $S$-versions of flat modules and von Neumann regular rings. In this paper, we generalized the classical pure exact sequences and absolutely pure modules to $u$-$S$-pure $u$-$S$-exact sequences and $u$-$S$-absolutely pure modules, and then obtain uniformly $S$-versions of some classical characterizations of pure exact sequences and absolutely pure modules (see Theorem \ref{c-s-pure} and Theorem \ref{c-s-abp}). Finally, we characterize uniformly $S$-von Neumann regular rings and  uniformly $S$-Noetherian rings using $u$-$S$-absolutely pure modules (see Theorem \ref{sabs-svnr} and Theorem \ref{usnoe-s-abp}). As our work involves the  uniformly $S$-torsion theory, we provide a quick review as below.

Recall from \cite{zwz21}, an $R$-module $T$ is said to be  $u$-$S$-torsion (with respect to $s$) provided that there exists an element $s\in S$ such that $sT=0$. An $R$-sequence $\cdots\rightarrow A_{n-1}\xrightarrow{f_n} A_{n}\xrightarrow{f_{n+1}} A_{n+1}\rightarrow\cdots$ is $u$-$S$-exact, if for any $n$ there is an element $s\in S$ such that $s\Ker(f_{n+1})\subseteq \Im(f_n)$ and $s\Im(f_n)\subseteq \Ker(f_{n+1})$.
An $R$-sequence $0\rightarrow A\xrightarrow{f} B\xrightarrow{g} C\rightarrow 0$ is called a short $u$-$S$-exact sequence (with respect to $s$), if $s\Ker(g)\subseteq \Im(f)$ and $s\Im(f)\subseteq \Ker(g)$ for some $s\in S$. An $R$-homomorphism $f:M\rightarrow N$ is an \emph{$u$-$S$-monomorphism} $($resp.,   \emph{$u$-$S$-epimorphism}, \emph{$u$-$S$-isomorphism}$)$  (with respect to $s$) provided $0\rightarrow M\xrightarrow{f} N$   $($resp., $M\xrightarrow{f} N\rightarrow 0$, $0\rightarrow M\xrightarrow{f} N\rightarrow 0$ $)$ is  $u$-$S$-exact  (with respect to $s$).
Suppose $M$ and $N$ are $R$-modules. We say $M$ is $u$-$S$-isomorphic to $N$ if there exists a $u$-$S$-isomorphism $f:M\rightarrow N$. A family $\C$  of $R$-modules  is said to be closed under $u$-$S$-isomorphisms if $M$ is $u$-$S$-isomorphic to $N$ and $M$ is in $\C$, then $N$ is  also in  $\C$. One can deduce from the following Poposition \ref{s-iso-inv} that the existence of $u$-$S$-isomorphisms of two $R$-modules is actually an equivalence relation.

\begin{proposition}\label{s-iso-inv}
Let $R$ be a ring and  $S$ a multiplicative subset of $R$. Suppose there is a $u$-$S$-isomorphism $f:M\rightarrow N$ for  $R$-modules  $M$ and $N$. Then there is a $u$-$S$-isomorphism $g:N\rightarrow M$ and $t\in S$ such that $f\circ g=t\Id_N$ and $g\circ f=t\Id_M$.
\end{proposition}
\begin{proof} Consider the following commutative diagram:
$$\xymatrix{
0\ar[r]^{}& \Ker(f)\ar[r]^{} &M\ar[rr]^{f}\ar@{->>}[rd] &&N \ar[r]^{} & \Coker(f)\ar[r]^{} &  0\\
  & & &\Im(f) \ar@{^{(}->}[ru] &&  &   \\}$$
with $s\Ker(f)=0$ and $sN\subseteq \Im(f)$ for some $s\in S$. Define $g_1:N\rightarrow \Im(f)$ where $g_1(n)=sn$ for any $n\in N$. Then $g_1$ is a well-defined $R$-homomorphism since $sn\in \Im(f)$. Define $g_2:\Im(f)\rightarrow M$ where $g_2(f(m))=sm$. Then  $g_2$ is  well-defined $R$-homomorphism. Indeed, if $f(m)=0$, then $m\in \Ker(f)$ and so $sm=0$. Set $g=g_2\circ g_1:N\rightarrow M$. We claim that $g$ is a $u$-$S$-isomorphism. Indeed, let $n$ be an element in $\Ker(g)$. Then $sn=g_1(n)\in \Ker(g_2)$. Note that $s\Ker(g_2)=0$. Thus $s^2n=0$. So $s^2\Ker(g)=0$. On the other hand, let $m\in M$. Then $g(f(m))=g_2\circ g_1(f(m))=g_2(f(sm))=s^2m$. Set $t=s^2\in S$. Then $g\circ f=t\Id_M$ and $tm\in \Im(g)$. So $tM\subseteq \Im(g)$. It follows that $g$ is  a $u$-$S$-isomorphism. It is also easy to verify that  $f\circ g=t\Id_N$.
\end{proof}

\begin{remark}
Let $R$ be a ring,  $S$ a multiplicative subset of $R$ and $M$ and $N$ $R$-modules. Then the condition ``there is an $R$-homomorphism $f:M\rightarrow N$  such that $f_S:M_S\rightarrow N_S$ is an isomorphism''  does not mean ``there is an $R$-homomorphism $g:N\rightarrow M$  such that $g_S:N_S\rightarrow M_S$ is an isomorphism''.

Indeed, let $R=\mathbb{Z}$ be the ring of integers, $S=R-\{0\}$  and $\mathbb{Q}$ the quotient field of integers. Then the embedding map $f:\mathbb{Z}\hookrightarrow \mathbb{Q}$ satisfies $f_S:\mathbb{Q}\rightarrow \mathbb{Q}$ is an isomorphism. However, since $\Hom_\mathbb{Z}(\mathbb{Q},\mathbb{Z})=0$, there does not exist any $R$-homomorphism $g:\mathbb{Q}\rightarrow \mathbb{Z}$  such that $g_S:\mathbb{Q}\rightarrow \mathbb{Q}$ is an isomorphism.
\end{remark}

The following two results state that a short $u$-$S$-exact sequence induces long $u$-$S$-exact sequences by the functors  ``Tor'' and  ``Ext'' as the classical cases.

\begin{theorem}\label{s-iso-tor}
Let $R$ be a ring,  $S$ a multiplicative subset of $R$ and $N$  an $R$-module. Let $0\rightarrow A\xrightarrow{f} B\xrightarrow{g} C\rightarrow 0$ be a $u$-$S$-exact sequence of $R$-modules. Then for any $n\geq 1$ there is an $R$-homomorphism $\delta_n:\Tor_{n}^R(C,N)\rightarrow \Tor_{n-1}^R(A,N)$ such that  the induced sequence
$$\cdots \rightarrow \Tor_{n}^R(A,N)\rightarrow \Tor_{n}^R(B,N)\rightarrow \Tor_{n}^R(C,N)\xrightarrow{\delta_n} \Tor_{n-1}^R(A,N)\rightarrow $$
$$\Tor_{n-1}^R(B,N)\rightarrow \cdots  \rightarrow\Tor_{1}^R(C,N)\xrightarrow{\delta_1} A\otimes_RN\rightarrow B\otimes_RN\rightarrow C\otimes_RN\rightarrow 0$$
 is $u$-$S$-exact.
\end{theorem}
\begin{proof}
Since the sequence $0\rightarrow A\xrightarrow{f} B\xrightarrow{g} C\rightarrow 0$ is $u$-$S$-exact at $B$. There are three exact sequences $0\rightarrow \Ker(f)\xrightarrow{i_{\Ker(f)}} A\xrightarrow{\pi_{\Im(f)}} \Im(f)\rightarrow 0$, $0\rightarrow \Ker(g)\xrightarrow{i_{\Ker(g)}} B\xrightarrow{\pi_{\Im(g)}} \Im(g)\rightarrow 0$ and $0\rightarrow \Im(g)\xrightarrow{i_{\Im(g)}} C\xrightarrow{\pi_{\Coker(g)}} \Coker(g)\rightarrow 0$ with $\Ker(f)$ and $\Coker(g)$ $u$-$S$-torsion. There also exists  $s\in S$ such that $s\Ker(g)\subseteq \Im(f)$ and  $s\Im(f)\subseteq \Ker(g)$. Denote $T=\Ker(f)$ and $T'=\Coker(g)$.

Firstly,  consider the exact sequence
$$\Tor_{n+1}^R(T',N)\rightarrow \Tor_{n}^R(\Im(g),N)\xrightarrow{\Tor_{n}^R(i_{\Im(g)},N)} \Tor_{n}^R(C,N)\rightarrow \Tor_{n}^R(T',N).$$ Since $T'$ is $u$-$S$-torsion, $\Tor_{n+1}^R(T',N)$ and $\Tor_{n}^R(T',N)$  is $u$-$S$-torsion. Thus $\Tor_{n}^R(i_{\Im(g)},N)$ is a $u$-$S$-isomorphism. So there is also a $u$-$S$-isomorphism $h^n_{\Im(g)}: \Tor_{n}^R(C,N)\rightarrow \Tor_{n}^R(\Im(g),N)$ by Proposition \ref{s-iso-inv}. Consider the exact sequence:
$$ \Tor_{n-1}^R(T,N)\rightarrow \Tor_{n-1}^R(A,N)\xrightarrow{\Tor_{n-1}^R(\pi_{\Im(f)},N)} \Tor_{n-1}^R(\Im(f),N)\rightarrow  \Tor_{n-2}^R(T,N).$$ Since $T$ is $u$-$S$-torsion, we have $\Tor_{n-1}^R(\pi_{\Im(f)},N)$ is a $u$-$S$-isomorphism. So there is also a $u$-$S$-isomorphism $h^{n-1}_{\Im(f)}: \Tor_{n-1}^R(\Im(f),N)\rightarrow \Tor_{n-1}^R(A,N)$ by Proposition \ref{s-iso-inv}.
We have two  exact sequences$$\Tor_{n+1}^R(T_1,N)\rightarrow \Tor_{n}^R(s\Ker(g),N)\xrightarrow{\Tor_{n}^R(i^1_{s\Ker(g)},N)}  \Tor_{n}^R(\Im(f),N)\rightarrow  \Tor_{n+1}^R(T_1,N)$$ and $$\Tor_{n+1}^R(T_2,N)\rightarrow \Tor_{n}^R(s\Ker(g),N)\xrightarrow{\Tor_{n}^R(i^2_{s\Ker(g)},N)}  \Tor_{n}^R(\Ker(g),N)\rightarrow  \Tor_{n+1}^R(T_2,N),$$  where $T_1=\Im(f)/s\Ker(g)$ and $T_2=\Im(f)/s\Im(f)$ is $u$-$S$-torsion. So $\Tor_{n}^R(i^1_{s\Ker(g)},N)$ and $\Tor_{n}^R(i^2_{s\Ker(g)},N)$ are $u$-$S$-isomorphisms. Thus there is a $u$-$S$-isomorphism $h^n_{s\Ker(g)}: \Tor_{n}^R(\Ker(g),N)\rightarrow \Tor_{n}^R(s\Ker(g),N)$.
Note that there is an exact sequence
$$\Tor_{n}^R(B,N)\xrightarrow{\Tor_{n}^R(\pi_{\Im(g)},N)}\Tor_{n}^R(\Im(g),N)\xrightarrow{\delta^{n}_{\Im(g)}} \Tor_{n-1}^R(\Ker(g),N)\xrightarrow{\Tor_{n-1}^R(i_{\Ker(g)},N)} \Tor_{n-1}^R(B,N).$$
Set $\delta_n=h^n_{\Im(g)}\circ \delta^{n}_{\Im(g)}\circ h^n_{s\Ker(g)}\circ \Tor_{n}^R(i^1_{s\Ker(g)},N)\circ h^{n-1}_{\Im(f)} :\Tor_{n}^R(C,N)\rightarrow \Tor_{n-1}^R(A,N)$. Since $h^n_{\Im(g)}, \delta^{n}_{\Im(g)}, h^n_{s\Ker(g)}$ and $h^{n-1}_{\Im(f)}$ are $u$-$S$-isomorphisms, we have the sequence $\Tor_{n}^R(B,N)\rightarrow\Tor_{n}^R(C,N)\xrightarrow{\delta^{n}} \Tor_{n-1}^R(A,N)$ $\rightarrow \Tor_{n-1}^R(B,N)$ is $u$-$S$-exact.

Secondly, consider the exact sequence: $$\Tor_{n+1}^R(T,N)\rightarrow \Tor_{n}^R(A,N)\xrightarrow{\Tor_{n}^R(i_{\Im(f)},N)} \Tor_{n}^R(\Im(f),N)\rightarrow \Tor_{n}^R(T,N).$$
Since $T$ is $u$-$S$-torsion, $\Tor_{n}^R(i_{\Im(f)},N)$ is a $u$-$S$-isomorphism. Consider the exact sequences: $$\Tor_{n+1}^R(\Im(g),N)\rightarrow \Tor_{n}^R(\Ker(g),N)\xrightarrow{\Tor_{n}^R(i_{\Ker(g)},N)} \Tor_{n}^R(B,N)\rightarrow \Tor_{n}^R(\Im(g),N)$$
and $$\Tor_{n+1}^R(T',N)\rightarrow \Tor_{n}^R(\Im(g),N)\xrightarrow{\Tor_{n}^R(i_{\Im(g)},N)} \Tor_{n}^R(C,N)\rightarrow \Tor_{n}^R(T',N).$$
Since $T'$ is $u$-$S$-torsion, we have $\Tor_{n}^R(i_{\Im(g)},N)$ is a $u$-$S$-isomorphism.  Since $\Tor_{n}^R(i^1_{s\Ker(g)},N)$ and $\Tor_{n}^R(i^2_{s\Ker(g)},N)$ are $u$-$S$-isomorphisms as above, $\Tor_{n}^R(A,N)\rightarrow \Tor_{n}^R(B,N)\rightarrow \Tor_{n}^R(C,N)$ is $u$-$S$-exact at $\Tor_{n}^R(B,N)$.

Iterating the above steps,  we have the following $u$-$S$-exact sequence:
$$\cdots \rightarrow \Tor_{n}^R(A,N)\rightarrow \Tor_{n}^R(B,N)\rightarrow \Tor_{n}^R(C,N)\xrightarrow{\delta_n} \Tor_{n-1}^R(A,N)\rightarrow $$
$$\Tor_{n-1}^R(B,N)\rightarrow \cdots  \rightarrow\Tor_{1}^R(C,N)\xrightarrow{\delta_1} A\otimes_RN\rightarrow B\otimes_RN\rightarrow C\otimes_RN\rightarrow 0.$$
\end{proof}

Similar to the proof of Theorem \ref{s-iso-tor}, we can deduce  the following result.

\begin{theorem}\label{s-iso-ext}
Let $R$ be a ring, $S$ a multiplicative subset of $R$ and $M$ and $N$ $R$-modules. Suppose $0\rightarrow A\xrightarrow{f} B\xrightarrow{g} C\rightarrow 0$ is a $u$-$S$-exact sequence of $R$-modules. Then for any $n\geq 1$ there are $R$-homomorphisms $\delta_n:\Ext^{n-1}_R(M,C)\rightarrow \Ext^{n}_R(M,A)$ and $\delta^n: \Ext^{n-1}_R(A,N)\rightarrow \Ext^{n}_R(C,N)$ such that  the induced sequences
\begin{center}
$0\rightarrow \Hom_R(M,A)\rightarrow \Hom_R(M,B)\rightarrow \Hom_R(M,C)\xrightarrow{\delta_0} \Ext^{1}_R(M,A)\rightarrow \cdots \rightarrow\Ext^{n-1}_R(M,B)\rightarrow \Ext^{n-1}_R(M,C)\xrightarrow{\delta_n} \Ext^{n}_R(M,A)\rightarrow \Ext^{n}_R(M,B)\rightarrow \cdots$
\end{center}
and
\begin{center}
$0\rightarrow \Hom_R(C,N)\rightarrow \Hom_R(B,N)\rightarrow \Hom_R(A,N)\xrightarrow{\delta^0} \Ext^{1}_R(C,N)\rightarrow \cdots \rightarrow
 \Ext^{n-1}_R(B,N)\rightarrow \Ext^{n-1}_R(A,N)\xrightarrow{\delta^n} \Ext^{n}_R(C,N)\rightarrow  \Ext^{n}_R(B,N)\rightarrow \cdots$
 \end{center}
 are $u$-$S$-exact.
\end{theorem}

\section{$u$-$S$-pure $u$-$S$-exact sequences}
Recall from \cite{R79} that an exact sequence $0\rightarrow A\rightarrow B\rightarrow C\rightarrow 0$ is said to be pure provided that for any $R$-module $M$, the induced sequence $0\rightarrow M\otimes_RA\rightarrow M\otimes_RB\rightarrow M\otimes_RC\rightarrow 0$ is also exact. Now we introduce the uniformly $S$-version of pure exact sequences.

\begin{definition}\label{def-s-f} Let $R$ be a ring, $S$ a multiplicative subset of $R$.
A short $u$-$S$-exact sequence $0\rightarrow A\rightarrow B\rightarrow C\rightarrow 0$ is said to be  \emph{ $u$-$S$-pure} provided that for any $R$-module $M$, the induced sequence $0\rightarrow M\otimes_RA\rightarrow M\otimes_RB\rightarrow M\otimes_RC\rightarrow 0$ is also $u$-$S$-exact.
\end{definition}

Obviously, any  pure exact sequence is $u$-$S$-pure.  In \cite[34.5]{w}, there are many characterizations of pure exact sequences. The next result generalizes some of these characterizations to $u$-$S$-pure $u$-$S$-exact sequences.

\begin{theorem}\label{c-s-pure}
Let $0\rightarrow A\xrightarrow{f} B\xrightarrow{f'} C\rightarrow 0$  be a  short $u$-$S$-exact sequence of $R$-modules.
Then the following statements are equivalent:
\begin{enumerate}
\item $0\rightarrow A\xrightarrow{f} B\xrightarrow{f'} C\rightarrow 0$ is a $u$-$S$-pure $u$-$S$-exact sequence;
\item there exists an element $s\in S$ satisfying that if a system of equations $f(a_i)=\sum\limits^m_{j=1}r_{ij}x_j\ (i=1,\cdots,n)$ with $r_{ij} \in R$ and unknowns $x_1, \cdots, x_m$ has a solution in $B$, then the system of equations  $sa_i=\sum\limits^m_{j=1}r_{ij}x_j\ (i=1,\cdots,n)$ is solvable in $A$.
\item there exists an element $s\in S$ satisfying that for any given commutative diagram with $F$ finitely generated free and $K$ a finitely generated submodule of $F$, there exists a  homomorphism $\eta:F\rightarrow A$ such that $s\alpha=\eta i$;
 $$\xymatrix@R=20pt@C=25pt{
0\ar[r] &K\ar[d]_{\alpha}\ar[r]^{i}&F\ar@{-->}[ld]^{\eta}\ar[d]^{\beta}\\
 & A\ar[r]_{f} &B\\}$$
\item there exists an element $s\in S$ satisfying that for any finitely presented $R$-module $N$, the induced sequence $0\rightarrow\Hom_R(N,A)\rightarrow \Hom_R(N,B)\rightarrow \Hom_R(N,C)\rightarrow 0$ is  $u$-$S$-exact with respect to $s$.
\end{enumerate}
\end{theorem}
\begin{proof} $(1)\Rightarrow(2)$: Set  $\Gamma=\{(K,R^n)\mid K$ is a finitely generated submodule of $ R^n$ and $n<\infty\}.$ Define $M=\bigoplus\limits_{(K,R^n)\in \Gamma}R^n/K$. Then  $0\rightarrow M\otimes_RA\xrightarrow{1\otimes f} M\otimes_RB\rightarrow M\otimes_RC\rightarrow 0$ by $(1)$.  So there is an element $s\in S$ such that $s\Ker(1_M\otimes f)=0$. Hence $s\Ker(1_{R^n/K}\otimes f)=0$ for any $(K,R^n)\in \Gamma$. Now assume that there exists $b_j\in B$ such that $f(a_i)=\sum\limits^m_{j=1}r_{ij}b_j$ for any $j=1,\cdots,m$. Let $F$ be a free $R$-module  with basis $\{e_1,\cdots,e_n\}$, and let $K\subseteq F$ be the submodule generated by $m$ elements $\{\sum\limits^n_{i=1}r_{ij}e_i\mid j=1,\cdots,m\}$. Then, $F/K$ is generated by $\{e_1+K,\cdots,e_n+K\}$. Note that $\sum\limits^n_{i=1}r_{ij}(e_i+K)=\sum\limits^n_{i=1}r_{ij}e_i+K=0+K$ in $F/K$. Hence, we have $$\sum\limits^n_{i=1}((e_i+K)\otimes f(a_i))=\sum\limits^n_{i=1}((e_i+K)\otimes (\sum\limits^m_{j=1}r_{ij}b_j))=\sum\limits^m_{j=1}((\sum\limits^n_{i=1}r_{ij}(e_i+K))\otimes b_j)=0$$
in $F/K\otimes B$. And so $\sum\limits^n_{i=1}((e_i+K)\otimes a_i)\in \Ker(1_{F/K}\otimes f)$.
Hence, $s\sum\limits^n_{i=1}((e_i+K)\otimes a_i)=\sum\limits^n_{i=1}((e_i+K)\otimes sa_i)=0$ in $F/K\otimes_RA$. By \cite[Chapter I, Lemma 6.1]{FS01}, there exists $d_j\in A$ and $t_{ij}\in R$ such that $sa_i=\sum\limits^t_{k=1}l_{ik}d_k$ and $\sum\limits^n_{i=1}l_{ik}(e_i+K)=0$, and so  $\sum\limits^n_{i=1}l_{ik}e_i\in K$. Then there exists $t_{jk}\in R$ such that $\sum\limits^n_{i=1}l_{ik}e_i=\sum\limits^m_{j=1}t_{jk}(\sum\limits^n_{i=1}r_{ij}e_i)=\sum\limits^n_{i=1}(\sum\limits^m_{j=1}(t_{jk}r_{ij})e_i)$. Since $F$ is free, we have $l_{ik}=\sum\limits^m_{j=1}r_{ij}t_{jk}$. Hence
$$sa_i=\sum\limits^t_{k=1}l_{ik}d_k=\sum\limits^t_{k=1}(\sum\limits^m_{j=1}r_{ij}t_{jk})d_k=
\sum\limits^m_{j=1}r_{ij}(\sum\limits^t_{k=1}t_{jk}d_k)$$
with $\sum\limits^t_{k=1}t_{jk}d_k\in A$. That is, $sa_i=\sum\limits^m_{j=1}r_{ij}x_j$ is solvable in $A$.

$(2)\Rightarrow(1)$: Let $s\in S$ satisfying (2) and $M$ be an $R$-module. Then we have a $u$-$S$-exact sequence
$M\otimes_RA\xrightarrow{1\otimes f} M\otimes_RB\rightarrow M\otimes_RC\rightarrow 0$ by Theorem \ref{s-iso-tor}. We will show that $\Ker(1\otimes f)$ is  $u$-$S$-torsion. Let $\{\sum\limits^{n_\lambda}_{i=1}u^\lambda_i\otimes a^\lambda_i \mid \lambda\in \Lambda\}$ be the generators of $\Ker(1\otimes f)$. Then  $\sum\limits^{n_\lambda}_{i=1}u^\lambda_i\otimes f(a^\lambda_i)=0$ in $M\otimes_RB$ for each $\lambda\in \Lambda$. By \cite[Chapter I, Lemma 6.1]{FS01}, there exists $r^\lambda_{ij}\in R$ and $b^\lambda_j\in B$ such that $f(a^\lambda_i)=\sum\limits^{m_\lambda}_{j=1}r^\lambda_{ij}b^\lambda_j$ and $\sum\limits^{n_\lambda}_{i=1}u^\lambda_ir^\lambda_{ij}=0$ for each $\lambda\in \Lambda$. So $sa^\lambda_i=\sum\limits^{m_\lambda}_{j=1}r^\lambda_{ij}x^\lambda_j$  have a solution, say $a^\lambda_j$ in $A$ by (2). Then $$s(\sum\limits^{n_\lambda}_{i=1}u^\lambda_i\otimes a^\lambda_i)=\sum\limits^{n_\lambda}_{i=1}u^\lambda_i\otimes sa^\lambda_i=\sum\limits^{n_\lambda}_{i=1}u^\lambda_i\otimes (\sum\limits^{m_\lambda}_{j=1}r^\lambda_{ij}a^\lambda_j)=\sum\limits^{m_\lambda}_{j=1}((\sum\limits^{n_\lambda}_{i=1}r^\lambda_{ij}u^\lambda_i)\otimes a^\lambda_j)=0$$ for each $\lambda\in \Lambda$. Hence $s\Ker(1\otimes f)=0$, and $0\rightarrow M\otimes_RA\rightarrow M\otimes_RB\rightarrow M\otimes_RC\rightarrow 0$ is  $u$-$S$-exact.

 $(2)\Rightarrow(3)$:  Let $s\in S$ satisfying (2) and $\{e_1,\cdots,e_n\}$ the basis of $F$. Suppose $K$ is generated by $\{y_i=\sum\limits^{m}_{j=1}r_{ij}e_j\mid i=1,\cdots,m\}$. Set $\beta(e_j)=b_j$ and $\alpha(y_i)=a_i$, then $f(a_i)=\sum\limits^{m}_{j=1}r_{ij}b_j$. By $(2)$, we have
$sa_i=\sum\limits^m_{j=1}r_{ij}d_j$ for some $d_j\in A$. Let $\eta:F\rightarrow A$ be $R$-homomorphism satisfying $\eta(e_j)=d_j$. Then $\eta i(y_i)=\eta i(\sum\limits^{m}_{j=1}r_{ij}e_j)=\sum\limits^{m}_{j=1}r_{ij}\eta(e_j)=\sum\limits^{m}_{j=1}r_{ij}d_j=sa_i=s\alpha(y_i)$,
and so we have $s\alpha=\eta i$.

$(3)\Rightarrow(4)$:  Let $s\in S$ satisfy (3). Note that  $A$ is $u$-$S$-isomorphic to $\Im(f)$ and $C$ is $u$-$S$-isomorphic to $\Coker(f)$. Thus, by Proposition \ref{s-iso-inv}, we have homomorphisms $t_1:A\rightarrow \Im(f)$ with $t_1(a)=f(a)$ for any $a\in A$ and $t'_1:\Im(f)\rightarrow A$ such that $t_1t'_1=s_1\Id_{\Im(f)}$ and $t'_1t_1=s_1\Id_{A}$, and  homomorphisms $t_2:\Coker(f)\rightarrow C$ and $t'_2:C\rightarrow \Coker(f)$ such that $f'=t_2\pi_{\Coker(f)}$, $t_2t'_2=s_2\Id_{C}$ and $t'_2t_2=s_2\Id_{\Coker(f)}$ for some $s_1,s_2\in S$ where $\pi_{\Coker(f)}:B\twoheadrightarrow \Coker(f)$ is the natural epimorphism.
Let $N$ be a  finitely presented $R$-module with $0\rightarrow K\rightarrow F\rightarrow N\rightarrow 0$ exact where $F$ is finitely generated free and $K$ finitely generated. Let $\gamma$ be a homomorphism in  $\Hom_R(N,C)$. Considering the exact sequence $0\rightarrow \Im(f)\rightarrow B\rightarrow \Coker(f)\rightarrow 0$, we have the following commutative diagram with rows exact:
$$\xymatrix@R=20pt@C=30pt{
0\ar[r] &K\ar[d]_{h}\ar[r]^{i_K}&F\ar[r]^{\pi_N} \ar[d]^{g}& N\ar[d]^{t'_2\gamma}\ar[r]&0\\
0\ar[r] & \Im(f)\ar[r]_{i_{\Im(f)}} &B\ar[r]_{\pi_{\Coker(f)}}  &\Coker(f)\ar[r] & 0\\}$$
By (3), there exists an homomorphism $\eta:F\rightarrow A$ such that $st'_1h=\eta i_K$. So $ss_1h=st_1t'_1h=t_1\eta i_K$. So the following diagram is also commutative:
$$\xymatrix@R=20pt@C=30pt{
0\ar[r] &K\ar[d]_{ss_1h}\ar[r]^{i_K}&F\ar[ld]^{t_1\eta} \ar[r]^{\pi_N} \ar[d]^{ss_1g}& N\ar@{-->}[ld]^{\delta}\ar[d]^{ss_1t'_2\gamma}\ar[r]&0\\
0\ar[r] & \Im(f)\ar[r]_{i_{\Im(f)}} &B\ar[r]_{\pi_{\Coker(f)}}  &\Coker(f)\ar[r] & 0\\}$$
So by  \cite[Exercise 1.60]{fk16}, there is an $R$-homomorphism $\delta:N\rightarrow B$ such that $ss_1t'_2\gamma=\pi_{\Coker(f)}\delta$. So $ss_1s_2\gamma=ss_1t_2t'_2\gamma=t_2\pi_{\Coker(f)}\delta=f'\delta=f'^{\ast}(\delta)$. Hence $f'^{\ast}: \Hom_R(N,B)\rightarrow \Hom_R(N,C)$ is a $u$-$S$-epimorphism  with respect to  $ss_1s_2$. Consequently, one can verify the  $R$-sequence $0\rightarrow \Hom_R(N,A)\rightarrow \Hom_R(N,B)\rightarrow \Hom_R(N,C)\rightarrow 0$ is $u$-$S$-exact  with respect to  $ss_1s_2$ by Theorem \ref{s-iso-ext}.

$(4)\Rightarrow(2)$:  Let $s\in S$ satisfying (4) and $0\rightarrow A\xrightarrow{f} B\xrightarrow{f'} C\rightarrow 0$ a  short $u$-$S$-exact sequence of $R$-modules. Similar with the proof of $(3)\Rightarrow(4)$, we have homomorphisms $t_1:A\rightarrow \Im(f)$ with $t_1(a)=f(a)$ for any $a\in A$ and $t'_1:\Im(f)\rightarrow A$ such that $t_1t'_1=s_1\Id_{\Im(f)}$ and $t'_1t_1=s_1\Id_{A}$, and  homomorphisms $t_2:\Coker(f)\rightarrow C$ and $t'_2:C\rightarrow \Coker(f)$ such that $f'=t_2\pi_{\Coker(f)}$, $t_2t'_2=s_2\Id_{C}$ and $t'_2t_2=s_2\Id_{\Coker(f)}$ for some $s_1,s_2\in S$ where $\pi_{\Coker(f)}:B\twoheadrightarrow \Coker(f)$ is the natural epimorphism.

Suppose that  $f(a_i)=\sum\limits^m_{j=1}r_{ij}b_j\ (i=1,\cdots,n)$ with $a_i\in A$, $b_j\in B$ and  $r_{ij}\in R$. Let $F_0$ be a free module with basis $\{e_1,\cdots,e_m\}$ and  $F_1$ a free module with basis $\{e'_1,\cdots,e'_n\}$. Then there are $R$-homomorphisms $\tau: F_0\rightarrow B$ and $\sigma: F_1\rightarrow \Im(f)$ satisfying $\tau(e_j)=b_j$ and $\sigma(e'_i)=f(a_i)$ for each $i,j$. Define $R$-homomorphism $h:F_1\rightarrow F_0$ satisfying $h(e'_i)=\sum\limits^m_{j=1}r_{ij}e_j$ for each $i$. Then $\tau h(e'_i)=\sum\limits^m_{j=1}r_{ij}\tau(e_j)=\sum\limits^m_{j=1}r_{ij}b_j=f(a_i)=\sigma(e'_i)$. Set $N=\Coker(h)$. Then $N$ is finitely presented. Thus there exists a homomorphism $\phi: N\rightarrow \Coker(f)$ such that the following diagram commutative:
$$\xymatrix@R=20pt@C=30pt{
 &F_1\ar[d]_{\sigma}\ar[r]^{h}&F_0\ar[r]^{g} \ar[d]^{\tau}& N\ar@{-->}[d]^{\phi}\ar[r]&0\\
0\ar[r] & \Im(f)\ar[r]_{i_{\Im(f)}} &B\ar[r]_{\pi_{\Coker(f)}}  &\Coker(f)\ar[r] & 0\\}$$
Note that the induced sequence $$0\rightarrow\Hom_R(N,\Im(f))\rightarrow \Hom_R(N,B)\rightarrow \Hom_R(N,\Coker(f))\rightarrow 0$$ is $u$-$S$-exact with respect to $s_1s_2s$ by (4). Hence there exists a homomorphism $\delta: N\rightarrow \Coker(f)$ such that $s_1s_2s\phi=\pi_{\Coker(f)}\delta$.
Consider the following commutative diagram:
$$\xymatrix@R=20pt@C=30pt{
 &F_1\ar[d]_{s_1s_2s\sigma}\ar[r]^{h}&F_0\ar@{-->}[ld]^{\eta}\ar[r]^{g} \ar[d]^{s_1s_2s\tau}& N\ar[ld]^{\delta}\ar[d]^{s_1s_2s\phi}\ar[r]&0\\
0\ar[r] & \Im(f)\ar[r]_{i_{\Im(f)}} &B\ar[r]_{\pi_{\Coker(f)}}  &\Coker(f)\ar[r] & 0\\}$$
We claim that there exists a homomorphism $\eta:F_0\rightarrow \Im(f)$ such that $\eta f=s_1s_2s\sigma$. Indeed, since $\pi_{\Coker(f)}\delta g=s_1s_2s\phi g=\pi_{\Coker(f)} s_1s_2s\tau$, we have $\Im(s_1s_2s\tau-\delta g)\subseteq \Ker(\pi_{\Coker(f)})=\Im(f)$. Define $\eta:F_0\rightarrow \Im(f)$ to be a homomorphism satisfying $\eta(e_i)=s_1s_2s\tau(e_i)-\delta g(e_i)$ for each $i$. So for each $e'_i\in F_1$, we have $\eta f(e'_i)=s_1s_2s\tau f(e'_i)-\delta g f(e'_i)=s_1s_2s\tau f(e'_i)$. Thus $i_{\Im(f)}(s_1s_2s\sigma)=s_1s_2si_{\Im(f)}\sigma=s_1s_2s\tau f=i_{\Im(f)}\eta f$. Therefore, $\eta f=s_1s_2s\sigma$. Hence $s_1s_2sf(a_i)=s_1s_2s\sigma(e'_i)=\eta f(e'_i)=\eta (\sum\limits^m_{j=1}r_{ij}e_j)=\sum\limits^m_{j=1}r_{ij}\eta (e_j)$ with $\eta (e_j)\in \Im(f)$. So we have $s^2_1s_2sa_i= s_1s_2st'_1f(a_i) =\sum\limits^m_{j=1}r_{ij}t'_1\eta (e_j)$ with $t'_1\eta (e_j)\in A$ for each $i$.
\end{proof}

Recall from  \cite[Definition 2.1]{zwz21-p} that a short $u$-$S$-exact sequence  $0\rightarrow A\xrightarrow{f} B\xrightarrow{g} C\rightarrow 0$ is said to be $u$-$S$-split provided that there are  $s\in S$ and $R$-homomorphism $t:B\rightarrow A$ such that $tf(a)=sa$ for any $a\in A$, that is, $t f=s\Id_A$.

\begin{proposition}\label{split-spur}
Let $\xi: 0\rightarrow A\xrightarrow{f} B\xrightarrow{g} C\rightarrow 0$ be an  $u$-$S$-split short $u$-$S$-exact sequence. Then $\xi$ is $u$-$S$-pure.
\end{proposition}
\begin{proof} Let $t:B\rightarrow A$ be an $R$-homomorphism satisfying $t f=s\Id_A$.  Let $f(a_i)=\sum\limits^m_{j=1}r_{ij}x_j$ be  a system of equations with $r_{ij} \in R$ and unknowns $x_1, \cdots, x_m$ has a solution, say $\{b_j\mid j=1,\mid, m\}$,  in $B$. Then  $sa_i=tf(a_i)=\sum\limits^m_{j=1}r_{ij}t(b_j)$ with $t(b_j)\in A$. Thus $sa_i=\sum\limits^m_{j=1}r_{ij}x_j$ is solvable in $A$. So $\xi$ is $u$-$S$-pure by Theorem \ref{c-s-pure}.
\end{proof}

Recall from \cite[Definition 3.1]{zwz21} that an $R$-module $F$ is called  $u$-$S$-flat provided that for any  $u$-$S$-exact sequence $0\rightarrow A\rightarrow B\rightarrow C\rightarrow 0$, the induced sequence $0\rightarrow A\otimes_RF\rightarrow B\otimes_RF\rightarrow C\otimes_RF\rightarrow 0$ is  $u$-$S$-exact. By \cite[Theorem 3.2]{zwz21}, an $R$-module $F$ is  $u$-$S$-flat if and only if   $\Tor^R_1(M,F)$ is   $u$-$S$-torsion for any  $R$-module $M$.

\begin{proposition}\label{s-f-sabp}
An $R$-module $F$ is $u$-$S$-flat if and only if every $(u$-$S$-$)$exact sequence $0\rightarrow A\rightarrow B\rightarrow F\rightarrow 0$ is $u$-$S$-pure.
\end{proposition}
\begin{proof} Suppose $F$ is a $u$-$S$-flat module. Let $M$ be an $R$-module and  $0\rightarrow A\rightarrow B\rightarrow F\rightarrow 0$ a short $u$-$S$-exact sequence. Then by Theorem \ref{s-iso-tor}, there is a $u$-$S$-exact sequence $\Tor_1^R(M,F)\rightarrow M\otimes_RA\rightarrow M\otimes_RB\rightarrow M\otimes_RF\rightarrow 0$. Since $F$ is $u$-$S$-flat, $\Tor^R_1(M,F)$ is   $u$-$S$-torsion by \cite[Theorem 3.2]{zwz21}. Hence $0\rightarrow M\otimes_RA\rightarrow M\otimes_RB\rightarrow M\otimes_RF\rightarrow 0$ is $u$-$S$-exact. So  $0\rightarrow A\rightarrow B\rightarrow F\rightarrow 0$ is $u$-$S$-pure.

On the other hand,  considering the exact sequence $0\rightarrow A\rightarrow P\rightarrow F\rightarrow 0$ with $P$ projective, we have an exact sequence  $0\rightarrow \Tor_1^R(M,F)\rightarrow M\otimes_RA\rightarrow M\otimes_RP\rightarrow M\otimes_RF\rightarrow 0$ for any $R$-module $M$. Since $0\rightarrow A\rightarrow P\rightarrow F\rightarrow 0$ is $u$-$S$-pure, $\Tor_1^R(M,F)$ is   $u$-$S$-torsion. So $F$ is $u$-$S$-flat
\end{proof}

\begin{proposition}\label{s-f-sp}
Let $\xi: 0\rightarrow A\rightarrow B\rightarrow C\rightarrow 0$ be a short $u$-$S$-exact sequence where $B$ is $u$-$S$-flat. Then $C$ is $u$-$S$-flat if and only if $\xi$ is $u$-$S$-pure.
\end{proposition}
\begin{proof} Suppose $C$ is $u$-$S$-flat. Then $\xi$ is $u$-$S$-pure by Proposition \ref{s-f-sabp}.

On the other hand, let $M$ be an $R$-module. Then we have a $u$-$S$-exact sequence $\Tor_1^R(M,B) \rightarrow\Tor_1^R(M,C) \rightarrow M\otimes_RA\rightarrow M\otimes_RB\rightarrow M\otimes_RC\rightarrow 0$.
Since  $B$ is $u$-$S$-flat, $\Tor_1^R(M,B)$ is  $u$-$S$-torsion by \cite[Theorem 3.2]{zwz21}. Since $\xi$ is $u$-$S$-pure by assumption, $0\rightarrow M\otimes_RA\rightarrow M\otimes_RB\rightarrow M\otimes_RF\rightarrow 0$ is $u$-$S$-exact. Then  $\Tor_1^R(M,C)$ is also  $u$-$S$-torsion. Thus $C$ is $u$-$S$-flat by \cite[Theorem 3.2]{zwz21} again.
\end{proof}

\section{uniformly $S$-absolutely pure modules}

Recall from \cite{M67} that an $R$-module $E$ is said to be absolutely pure provided that $E$ is a pure submodule of every module which contains $E$ as a submodule, that is, any short exact sequence $0\rightarrow E\rightarrow B\rightarrow C\rightarrow 0$ beginning with $E$ is pure. Now we give the uniformly $S$-analogue of absolutely pure modules.

\begin{definition}\label{s-abp} Let $R$ be a ring and $S$ a multiplicative subset of $R$.
An $R$-module $E$ is said to be $u$-$S$-absolutely pure (abbreviates uniformly $S$-absolutely pure) provided that any short $u$-$S$-exact sequence $0\rightarrow E\rightarrow B\rightarrow C\rightarrow 0$ beginning with $E$ is $u$-$S$-pure.
\end{definition}

Recall from \cite[Definition 4.1]{QKWCZ21} that an $R$-module $E$ is called  \emph{$u$-$S$-injective} provided that the induced sequence $$0\rightarrow \Hom_R(C,E)\rightarrow \Hom_R(B,E)\rightarrow \Hom_R(A,E)\rightarrow 0$$ is $u$-$S$-exact for any $u$-$S$-exact sequence $0\rightarrow A\rightarrow B\rightarrow C\rightarrow 0$. Following from \cite[Theorem 4.3]{QKWCZ21}, an $R$-module $E$ is  $u$-$S$-injective if and only if for any short exact sequence $0\rightarrow A\rightarrow B\rightarrow C\rightarrow 0$, the induced sequence $0\rightarrow \Hom_R(C,E)\rightarrow \Hom_R(B,E)\rightarrow \Hom_R(A,E)\rightarrow 0$ is  $u$-$S$-exact, if and only if  $\Ext_R^1(M,E)$ is   $u$-$S$-torsion for any  $R$-module $M$, if and only if  $\Ext_R^n(M,E)$ is   $u$-$S$-torsion for any  $R$-module $M$ and $n\geq 1$. Next, we characterize $u$-$S$-absolutely pure modules in terms of $u$-$S$-injective modules.

\begin{theorem}\label{c-s-abp}
Let $R$ be a ring, $S$ a multiplicative subset of $R$ and $E$ an $R$-module.
Then the following statements are equivalent:
\begin{enumerate}
\item $E$ is $u$-$S$-absolutely pure;
\item  any short exact sequence $0\rightarrow E\rightarrow B\rightarrow C\rightarrow 0$ beginning with $E$ is $u$-$S$-pure;
\item  $E$ is a $u$-$S$-pure submodule in every $u$-$S$-injective module containing $E$;
\item  $E$ is a $u$-$S$-pure submodule in every injective module containing $E$;
\item  $E$ is a $u$-$S$-pure submodule in its injective envelope;
\item there exists an element $s\in S$ satisfying that  for any finitely presented $R$-module $N$, $\Ext_R^1(N,E)$ is  $u$-$S$-torsion with respect to $s$;
\item  there exists an element $s\in S$ satisfying that  if  $P$ is finitely generated projective, $K$ a finitely generated submodule of $P$ and  $f:K\rightarrow E$ is an  $R$-homomorphism, then there is an $R$-homomorphism $g:P\rightarrow E$ such that $sf=gi$.
\end{enumerate}
\end{theorem}

\begin{proof} $(1)\Rightarrow (2)\Rightarrow (3)\Rightarrow (4)\Rightarrow (5)$: Trivial.

$(5)\Rightarrow(6)$: Let $I$ be the injective envelope of $E$. Then we have a $u$-$S$-pure exact sequence $0\rightarrow E\rightarrow I\rightarrow L\rightarrow 0$ by (5). Then, by  Theorem \ref{c-s-pure}, there is an element $s\in S$ such that $0\rightarrow\Hom_R(N,E)\rightarrow \Hom_R(N,I)\rightarrow \Hom_R(N,L)\rightarrow 0$  is $u$-$S$-exact with respect to $s$  for any finitely presented $R$-module $N$. Since  $0\rightarrow\Hom_R(N,E)\rightarrow \Hom_R(N,I)\rightarrow \Hom_R(N,L)\rightarrow \Ext_R^1(N,E)\rightarrow 0$ is exact. Hence $\Ext_R^1(N,E)$ is  $u$-$S$-torsion with respect to $s$ for any finitely presented $R$-module $N$.

$(6)\Rightarrow(1)$: Let $s\in S$ satisfy (6). Let  $N$ be a finitely presented $R$-module and $0\rightarrow E\rightarrow B\rightarrow C\rightarrow 0$ a $u$-$S$-exact sequence with respect to $s_1\in S$. Then, by Theorem \ref{s-iso-ext},  there is a $u$-$S$-exact sequence $0\rightarrow\Hom_R(N,E)\rightarrow \Hom_R(N,B)\rightarrow \Hom_R(N,C)\rightarrow \Ext_R^1(N,E)$ with respect to $s_1$ for any finitely presented $R$-module $N$. By (6), $0\rightarrow\Hom_R(N,E)\rightarrow \Hom_R(N,B)\rightarrow \Hom_R(N,C)\rightarrow  0$  is $u$-$S$-exact with respect to $ss_1$ for any finitely presented $R$-module $N$ .
Hence $E$ is $u$-$S$-absolutely pure by  Theorem \ref{c-s-pure}.

$(6)\Rightarrow(7)$:  Let $s\in S$ satisfy (6). Considering the exact sequence $0\rightarrow K\xrightarrow{i} P\rightarrow P/K\rightarrow 0$, we have the following exact sequence $ \Hom_R(P,E)\xrightarrow{i_{\ast}} \Hom_R(K,E)\rightarrow \Ext_R^1(P/K,E)\rightarrow 0$. Since $P/K$ is finitely presented, $\Ext_R^1(P/K,E)$ is $u$-$S$-torsion with respect to $s$ by $(6)$. Hence $i_{\ast}$ is a $u$-$S$-epimorphism, and so $s\Hom_R(K,E)\subseteq \Im(i_{\ast})$. Let $f:K\rightarrow E$ be an $R$-homomorphism. Then there is an $R$-homomorphism $g:P\rightarrow E$ such that $sf=gi$.

$(7)\Rightarrow(6)$: Let $s\in S$ satisfy (7). Let $N$ be a finitely presented $R$-module.  Then we have an exact sequence $0\rightarrow K\xrightarrow{i} P\rightarrow N\rightarrow 0$ where $P$ is finitely generated projective and $K$ is finitely generated. Consider the following exact sequence $ \Hom_R(P,E)\xrightarrow{i_{\ast}} \Hom_R(K,E)\rightarrow \Ext_R^1(N,E)\rightarrow 0$.  By $(7)$, we have  $s\Hom_R(K,E)\subseteq \Im(i_{\ast})$. Hence $\Ext_R^1(N,E)$ is  $u$-$S$-torsion with respect to $s$.
\end{proof}

\begin{proposition}\label{s-inj-prop}
Let $R$ be a ring and $S$ a multiplicative subset of $R$. Then the following statements hold.
\begin{enumerate}
\item  Any absolutely pure module and any  $u$-$S$-injective module is $u$-$S$-absolutely pure.
\item Any finite direct sum of  $u$-$S$-absolutely pure modules is  $u$-$S$-absolutely pure.
\item Let $0\rightarrow A\xrightarrow{f} B\xrightarrow{g} C\rightarrow 0$  be a $u$-$S$-exact sequence. If $A$ and $C$ are  $u$-$S$-absolutely pure modules, so is $B$.
\item  The class of $u$-$S$-absolutely pure modules is closed under $u$-$S$-isomorphisms.

\item Let $0\rightarrow A\rightarrow B\rightarrow C\rightarrow 0$  be a $u$-$S$-pure $u$-$S$-exact sequence. If $B$ is  $u$-$S$-absolutely pure, so is $B$.
\end{enumerate}
\end{proposition}
\begin{proof}
$(1)$  Follows from Theorem \ref{c-s-abp}.

$(2)$ Suppose $E_1,\cdots,E_n$ are  $u$-$S$-absolutely pure modules.  Then there exists  $s_i\in S$ such that  $s_i\Ext_R^1(M,E_i)=0$ for any  finitely presented $R$-module $M$ $(i=1,\cdots,n)$. Set $s=s_1\cdots s_n$. Then $s\Ext_R^1(M,\bigoplus\limits_{i=1}^n E_i)\cong \bigoplus\limits_{i=1}^ns\Ext_R^1(M, E_i)=0$. Thus $\bigoplus\limits_{i=1}^n E_i$ is  $u$-$S$-absolutely pure.

$(3)$ Let $0\rightarrow A\xrightarrow{f} B\xrightarrow{g} C\rightarrow 0$  be a $u$-$S$-exact sequence. Since $A$ and $C$ are  $u$-$S$-absolutely pure modules, then,  by Theorem \ref{c-s-abp}, $\Ext^{1}_R(N,A)$ and $\Ext^{1}_R(N,C)$ are  $u$-$S$-torsion with respect to some $s_1, s_2\in S$ for any finitely presented $R$-module $N$. Considering the $u$-$S$-sequence $\Ext^{1}_R(N,A)\rightarrow \Ext^{1}_R(N,B)\rightarrow \Ext^{1}_R(N,C)$ by Theorem \ref{s-iso-ext}, we have   $\Ext^{1}_R(N,B)$ is   $u$-$S$-torsion with respect to $s_1s_2$ for any finitely presented $R$-module $N$. Hence  $B$ is  $u$-$S$-absolutely pure by Theorem \ref{c-s-abp} again.

$(4)$ Considering the $u$-$S$-exact sequences $0\rightarrow A\rightarrow B\rightarrow 0\rightarrow 0$ and  $0 \rightarrow 0\rightarrow  A\rightarrow B\rightarrow 0$, we have $A$ is $u$-$S$-absolutely pure if and only if  $B$ is  $u$-$S$-absolutely pure by $(3)$.

$(5)$  Let $0\rightarrow A\rightarrow B\rightarrow C\rightarrow 0$  be a $u$-$S$-pure $u$-$S$-exact sequence with respect to some $s\in S$. Then,  by Theorem \ref{s-iso-ext}, there exists a $u$-$S$-sequence $0\rightarrow \Hom_R(N,A)\rightarrow \Hom_R(N,B)\rightarrow \Hom_R(N,C)\rightarrow \Ext^{1}_R(N,A)\rightarrow \Ext^{1}_R(N,B)$ with respect to  $s$ for any finitely presented $R$-module $N$. Note that the natural homomorphism $\Hom_R(N,B)\rightarrow \Hom_R(N,C)$ is a $u$-$S$-epimorphism. Since $B$ is  $u$-$S$-absolutely pure, it follows that $\Ext^{1}_R(N,B)$ is  $u$-$S$-torsion  with respect to some $s_1\in S$ for any finitely presented $R$-module $N$ by Theorem \ref{c-s-abp}. Then $\Ext^{1}_R(N,A)$ is  $u$-$S$-torsion  with respect to  $ss_1$  for any finitely presented $R$-module $N$. Thus $A$ is  $u$-$S$-absolutely pure by Theorem \ref{c-s-abp} again.
\end{proof}

Let $\p$ be a prime ideal of $R$. We say an $R$-module $E$ is \emph{$u$-$\p$-absolutely pure} shortly provided that  $E$ is  $u$-$(R\setminus\p)$-absolutely pure.
\begin{proposition}\label{s-flat-loc-char}
Let $R$ be a ring and $E$ an $R$-module. Then the following statements are equivalent:
 \begin{enumerate}
\item  $E$ is absolutely pure;
\item   $E$ is   $u$-$\p$-absolutely pure for any $\p\in \Spec(R)$;
\item   $E$ is  $u$-$\m$-absolutely pure for any $\m\in \Max(R)$.
 \end{enumerate}
\end{proposition}
\begin{proof} $(1)\Rightarrow (2)\Rightarrow (3):$  Trivial.

$(3)\Rightarrow (1):$  Since $E$ is  $\m$-absolutely pure for any $\m\in \Max(R)$, we have $\Ext_R^1(N,E)$ is uniformly $(R\setminus\m)$-torsion for any finitely presented $R$-module $N$. Thus for any $\m\in \Max(R)$, there exists  $s_{\m}\in S$ such that $s_{\m}\Ext_R^1(N,E)=0$ for any finitely presented $R$-module $N$. Since the ideal generated by all $s_{\m}$ is $R$, $\Ext_R^1(N,E)=0$ for any finitely presented $R$-module $N$. So $E$ is absolutely pure.
\end{proof}

Recall from \cite[Definition 3.12]{zwz21} a ring $R$ is called  uniformly $S$-von Neumann regular provided there exists an element  $s\in S$ satisfies that for any $a\in R$ there exists  $r\in R$ such that $sa=ra^2$. It was proved in  \cite[Theorem 3.13]{zwz21} that a ring $R$ is  uniformly $S$-von Neumann regular if and only if any $R$-module is $u$-$S$-flat.
\begin{theorem}\label{sabs-svnr}
A ring $R$ is  uniformly $S$-von Neumann regular if and only if any $R$-module is $u$-$S$-absolutely pure.
\end{theorem}
\begin{proof} Suppose $R$ is an  uniformly $S$-von Neumann regular ring. Let $M$ be an $R$-module and $I$ its injective envelope. Then $M/I$ is $u$-$S$-flat by \cite[Theorem 3.13]{zwz21}. Hence $M$ is a $u$-$S$-pure submodule of $I$ by Proposition \ref{s-f-sabp}. So $M$ is $u$-$S$-absolutely pure by Theorem \ref{c-s-abp}.

On the other hand, let $M$ be an $R$-module and $\xi: 0\rightarrow K\rightarrow P\rightarrow M\rightarrow 0$ an exact sequence with $P$ projective. Then $P$ is  $u$-$S$-flat. Since $K$ is  $u$-$S$-absolutely pure, the exact sequence $\xi$ is $u$-$S$-pure. By Proposition \ref{s-f-sp}, $M$ is also $u$-$S$-flat. Hence $R$ is  uniformly $S$-von Neumann regular by  \cite[Theorem 3.13]{zwz21}.
\end{proof}

It follows from Proposition \ref{s-inj-prop} that every absolutely pure module is $u$-$S$-absolutely pure. The following example shows that the converse is not true in general
\begin{example} \cite[Example 3.18]{zwz21}
Let $T=\mathbb{Z}_2\times \mathbb{Z}_2$ be a semi-simple ring and $s=(1,0)\in T$. Then any element $a\in T$ satisfies $a^2=a$ and $2a=0$. Let $R=T[x]/\langle sx,x^2\rangle$ with $x$ the indeterminate  and $S=\{1,s\}$ be a multiplicative subset of $R$. Then $R$ is an uniformly $S$-von Neumann regular ring, but $R$ is not von Neumann regular. Thus there exists a $u$-$S$-absolutely pure module $M$ which is not absolutely pure by Theorem \ref{sabs-svnr}.
\end{example}

Recall from \cite{QKWCZ21} that a ring $R$ is called a  uniformly $S$-Noetherian ring  provided that there exists an element $s\in S$ such that for any ideal $J$ of $R$, $sJ \subseteq K$ for some finitely generated sub-ideal $K$ of $J$. Following from Theorem \cite[Theorem 4.10]{QKWCZ21} that if  $S$ is a regular multiplicative subset of $R$ (i.e., the multiplicative set $S$ is composed of non-zero-divisors), then $R$ is   uniformly $S$-Noetherian if and only if  any direct sum of injective modules is $u$-$S$-injective. Now we give a new characterization of  uniformly $S$-Noetherian rings.

\begin{theorem}\label{usnoe-s-abp}
Let $R$ be a ring, $S$ a regular multiplicative subset of $R$.
Then the following statements are equivalent:
\begin{enumerate}
\item  $R$ is a uniformly   $S$-Noetherian ring;
\item    any $u$-$S$-absolutely pure module is $u$-$S$-injective;
\item    any absolutely pure module is $u$-$S$-injective.
\end{enumerate}
\end{theorem}
\begin{proof}
$(1)\Rightarrow (2)$: Suppose $R$ is a  uniformly $S$-Noetherian ring.  Let $s$ be an element in  $S$ such that for any ideal $J$ of $R$, $sJ \subseteq K$ for some finitely generated sub-ideal $K$ of $J$. Let $E$ be a $u$-$S$-absolutely pure module. Then there exists $s_2\in S$ such that $s_2\Ext_R^1(N,E)=0$ for any finitely presented $R$-module $N$. Let $s_1$ be an element in $S$. Consider the induced exact sequence $\Hom_R(R,E)\rightarrow \Hom_R(Rs_1,E)\rightarrow  \Ext_R^1(R/Rs_1,E)\rightarrow 0$. Since $R/Rs_1$ is finitely presented, $s_2\Ext_R^1(R/Rs_1,E)=s_2(E/s_1E)=0$ since $s_1$ is a non-zero-divisor. Then $s_2E=s_1s_2E$, and thus $s_2E$ is $u$-$S$-divisible. Since $s_2E$ is $u$-$S$-isomorphic to $E$, $s_2E$ is also $u$-$S$-absolutely pure by Proposition \ref{s-inj-prop}. Hence there exists  $s_3\in S$ such that $s_3\Ext_R^1(N,E)=0$ for any finitely presented $R$-module $N$. Consider the induced $u$-$S$-exact sequence $\Hom_R(J/K,s_2E)\rightarrow\Ext_R^1(R/J,s_2E)\rightarrow\Ext_R^1(R/K,s_2E)$. Since $R/K$ is finitely presented, we have $s_3\Ext_R^1(R/K,s_2E)=0$. Note that  $s\Hom_R(J/K,s_2E)=0$. Then $ss_3\Ext_R^1(R/J,s_2E)=0$. Since $s_2E$ is $u$-$S$-divisible, we have $s_2E$ is $u$-$S$-injective by \cite[Proposition 4.9]{QKWCZ21}. Since $s_2E$ is $u$-$S$-isomorphic to $E$, $E$ is also $u$-$S$-injective by \cite[Proposition 4.7]{QKWCZ21}.

$(2)\Rightarrow (3)$: Trivial.

$(3)\Rightarrow (1)$:  Let $\{I_{\lambda}\mid \lambda\in \Lambda\}$ be a family of injective modules, then $\bigoplus\limits_{\lambda\in \Lambda} I_{\lambda}$ is absolutely pure, and thus is $u$-$S$-injective by assumption. Consequently,  $R$ is a  uniformly $S$-Noetherian ring by \cite[Theorem 4.10]{QKWCZ21}.
\end{proof}

It is well-known that any direct sum and any direct product of absolutely pure modules are also absolutely pure. However, it does not work for $u$-$S$-absolutely pure modules.
\begin{example}
Let $R=\mathbb{Z}$ be the ring of integers, $p$ a prime in $\mathbb{Z}$ and  $S=\{p^n|n\geq 0\}$.
Then an $R$-module $M$ is $u$-$S$-absolutely pure module if and only if it is  $u$-$S$-injective by Theorem \ref{usnoe-s-abp}. Let $\mathbb{Z}/\langle p^k\rangle$ be cyclic group of order $p^k$ ($k\geq 1$). Then each $\mathbb{Z}/\langle p^k\rangle$ is   $u$-$S$-torsion, and thus is $u$-$S$-absolutely pure. However, the product $M:=\prod\limits_{k=1}^\infty \mathbb{Z}/\langle p^k\rangle$ is not $u$-$S$-injective by \cite[Remark 4.6]{QKWCZ21}, so it is also not $u$-$S$-absolutely pure.

We claim that  the direct sum $N:=\bigoplus\limits_{k=1}^\infty \mathbb{Z}/\langle p^k\rangle$ is also not $u$-$S$-absolutely pure. Indeed, consider the following exact sequence induced by the short exact sequence  $0\rightarrow \mathbb{Z}\rightarrow \mathbb{Q}\rightarrow \mathbb{Q}/\mathbb{Z}\rightarrow 0$:
$$0=\Hom_\mathbb{Z}(\mathbb{Q},N)\rightarrow \Hom_\mathbb{Z}(\mathbb{Z},N)\rightarrow \Ext^1_\mathbb{Z}(\mathbb{Q}/\mathbb{Z},N)\rightarrow \Ext^1_\mathbb{Z}(\mathbb{Q},N)\rightarrow 0.$$
Since the submodule $N=\Hom_\mathbb{Z}(\mathbb{Z},N)$ is not  $u$-$S$-torsion, $\Ext^1_\mathbb{Z}(\mathbb{Q}/\mathbb{Z},N)$ is also not  $u$-$S$-torsion. Then $N$ is not  $u$-$S$-injective by \cite[Theorem 4.3]{QKWCZ21}. So the direct sum $N:=\bigoplus\limits_{k=1}^\infty \mathbb{Z}/\langle p^k\rangle$ is also not $u$-$S$-absolutely pure.
\end{example}

We also note that, in Theorem \ref{c-s-abp}, the element $s\in S$ in the statement (6) (similar in  the statement (7)) is uniform for all finitely presented $R$-modules $N$.
\begin{example}
Let $R=\mathbb{Z}$ be the ring of integers, $p$ a prime in $\mathbb{Z}$ and  $S=\{p^n|n\geq 0\}$. Let $J_p$ be the additive group of all $p$-adic integers $($see \cite{FS15} for example$)$. Then $\Ext_{R}^1(N,J_p)$ is $u$-$S$-torsion for any finitely presented $R$-modules $N$. However, $J_p$ is not  $u$-$S$-absolutely pure.
\end{example}
\begin{proof} Let $N$ be a finitely presented $R$-module. Then, by \cite[Chapter 3, Theorem 2.7]{FS15}, $N\cong \mathbb{Z}^n\oplus \bigoplus\limits_{i=1}^m(\mathbb{Z}^n/\langle p^i\rangle)^{n_i}\oplus T$,
 where  $T$ is a finitely generated torsion $S$-divisible torsion-module. Thus $$\Ext_{R}^1(N,J_p)\cong \bigoplus\limits_{i=1}^m\Ext_{R}^1(\mathbb{Z}^n/\langle p^i\rangle,J_p)\cong \bigoplus\limits_{i=1}^m  (J_p/p^i J_p)\cong \bigoplus\limits_{i=1}^m\mathbb{Z}^n/\langle p^i\rangle$$ by \cite[Chapter 9, section 3 (G)]{FS15} and \cite[Chapter 1, Exercise 3(10)]{FS15}. So $\Ext_{R}^1(N,J_p)$ is obviously   $u$-$S$-torsion. However, $J_p$ is not  $u$-$S$-injective by \cite[Theorem 4.5]{QKWCZ21}. So  $J_p$ is not  $u$-$S$-absolutely pure by Theorem \ref{usnoe-s-abp}.
\end{proof}

\begin{acknowledgement}\quad\\
The author was supported by  the National Natural Science Foundation of China (No. 12061001).
\end{acknowledgement}

\end{document}